\documentclass[10pt, a4paper, reqno, oneside]{amsart}

\usepackage{amsmath, amsfonts, amsthm, amssymb, mathtools, enumerate, multicol, scalefnt, relsize}
\usepackage[mathscr]{euscript}
\usepackage{mathbbol}
\usepackage[latin1]{inputenc}
\usepackage{graphicx}
\usepackage[all]{xy}
\usepackage{enumitem}
\usepackage{autobreak,lipsum}
\setlist[itemize]{noitemsep, topsep=1pt, leftmargin=20pt}
\usepackage{xfrac}
\usepackage{todonotes}

\usepackage{fullpage}
\usepackage{hyperref}

\newcommand\bcdot{\ensuremath{
		\mathchoice
		{\mskip\thinmuskip\lower0.2ex\hbox{\scalebox{1.6}{$\cdot$}}\mskip\thinmuskip}}
	{\mskip\thinmuskip\lower0.2ex\hbox{\scalebox{1.6}{$\cdot$}}\mskip\thinmuskip}
	{\lower0.3ex\hbox{\scalebox{1.2}{$\cdot$}}}
	{\lower0.3ex\hbox{\scalebox{1.2}{$\cdot$}}}
}

\theoremstyle{plain}
\newtheorem{theo}{Theorem}[section]

\theoremstyle{definition}

\theoremstyle{plain}
\newtheorem{lemma}[theo]{Lemma}

\newtheorem{proposition}[theo]{Proposition}

\theoremstyle{definition}

\newtheorem{remark}[theo]{Remark}

\theoremstyle{plain}
\newtheorem{thmint}{Theorem}

\newtheorem{corint}[thmint]{Corollary}

\theoremstyle{definition}
\newtheorem*{definition*}{Definition}

\DeclareSymbolFontAlphabet{\mathbb}{AMSb}
\DeclareSymbolFontAlphabet{\mathbbl}{bbold}

\makeatletter
\@namedef{subjclassname@2020}{\textup{2020} Mathematics Subject Classification}
\makeatother

\allowdisplaybreaks

\title[]{The pluriclosed flow and the Vaisman condition}

\author{Giuseppe Barbaro}
\address[Giuseppe Barbaro]{Department of Mathematics, Aarhus University, Ny Munkegade 118, 8000 Aarhus C, Denmark}
\email{g.barbaro@math.au.dk}

\author{Francesco Pediconi}
\address[Francesco Pediconi]{Dipartimento di Matematica e Informatica ``Ulisse Dini'' \\ Universit\`a degli Studi di Firenze, viale Morgagni 67/a, 50134 Firenze, Italy}
\email{francesco.pediconi@unifi.it}

\author{Nicoletta Tardini}
\address[Nicoletta Tardini]{Dipartimento di Scienze Matematiche, Fisiche e Informatiche \\ Universit\`a degli Studi di Parma, Parco Area delle Scienze 53/a, 43124 Parma, Italy}
\email{nicoletta.tardini@unipr.it}

\thanks{All authors are members of GNSAGA of INdAM. The first-named author has been supported by Villum Young Investigator 0019098. The second-named and third-named authors have been supported by the PRIN 2022 project ``Real and Complex Manifolds: Geometry and Holomorphic Dynamics'' (code 2022AP8HZ9). The third-named author was supported by University of Parma through the action Bando di Ateneo 2023 per la ricerca.}

\begin{document}

\begin{abstract}
We prove that the pluriclosed flow preserves the Vaisman condition on compact complex surfaces if and only if the starting metric has constant scalar curvature.
\end{abstract}

\maketitle

\section{Introduction}
\setcounter{equation} 0

After the groundbreaking achievements of Hamilton and Perelman, geometric flows play a crucial role in Geometric Analysis and Differential Geometry. In the K\"ahler setting, the {\it K\"ahler Ricci flow} has been used to provide alternative proofs of major results, such as \cite{MR0799272, MR2535001}. Subsequently, Streets and Tian introduced the {\it pluriclosed flow} \cite{MR2673720} to apply these analytic techniques to the study of complex, non-K\"ahler manifolds. \smallskip

The pluriclosed flow evolves in the class of {\it pluriclosed metrics}, that are Hermitian metrics $\omega$ on a complex manifold $(M,J)$ satisfying ${\rm d}J{\rm d}\omega = 0$, via the parabolic equation
\begin{equation} \label{eq:PCflow}
\partial_t\, \omega(t) = -\rho(\omega(t))^{1,1} \,\, .
\end{equation}
Here, we recall that the {\it Bismut connection} of $(M,J,\omega)$ is the unique Hermitian connection with totally skew-symmetric torsion \cite{MR1006380} and $\rho(\omega)^{1,1}$ denotes the $(1,1)$-component of the {\it Bismut Ricci form} obtained by tracing the second two indexes of its curvature.

In complex dimension $2$, pluriclosed metrics are {\it Gauduchon}, hence by \cite[Th{\'e}or{\`e}me 1]{MR0470920}, they always exist in any conformal class on compact surfaces. This makes the pluriclosed flow particularly well-suited for studying the topology and geometry of compact non-K\"ahler surfaces, with important implications also in their geometrization \cite{MR4181011, MR3110582}. Up to now, the long-time behavior of the flow is well understood only in specific cases, in the presence of symmetries \cite{MR3511471, MR3957836, MR4687765, FLS24} or curvature restrictions \cite{MR3462132, MR4629758, Barb22, barbaro2023bismut}. Therefore, at this stage, it is natural to impose additional conditions to understand the pluriclosed flow solutions better. In this paper, we address the evolution of Vaisman metrics. \smallskip

We recall that, on a complex manifold $(M^{2n},J)$, we can associate to any Hermitian metric $\omega$ a $1$-form $\theta$, called {\it Lee form}, uniquely determined by the condition ${\rm d}\omega^{n-1} = \theta \wedge \omega^{n-1}$. The metric $\omega$ is called {\it Vaisman} if $\theta$ is Levi-Civita parallel, non-vanishing and ${\rm d}\omega = \frac1{n-1}\theta \wedge \omega$. We refer to \cite{MR4771164} and the references therein for an introduction to Vaisman metrics.

In complex dimension $n=2$, Vaisman metrics are pluriclosed. Moreover, compact Vaisman surfaces have been classified in \cite{MR1760667}: they are properly elliptic surfaces, Kodaira surfaces, and Hopf surfaces of class $1$. The first natural question is whether or not the flow preserves this class of metrics. In \cite{MR4257077}, it is proved that this condition is preserved if the initial Vaisman metric has constant scalar curvature. Furthermore, by \cite{MR4732911, MR4760212}, the pluriclosed flow preserves $\mathsf{T}^2$-invariant Vaisman metrics on the Kodaira-Thurston surface if and only if the initial metric has constant scalar curvature. Notice that, for Vaisman surfaces, the Riemannian scalar curvature and the {\it Bismut scalar curvature} differ by constants. So, the nomenclature {\it constant scalar curvature} indicates that both are constant. The main result of this paper is the following.

\begin{thmint} \label{thm:MAIN}
Let $(M^4, J, \omega_0)$ be a compact Vaisman surface and let $\{\omega(t)\}_{t \in [0,T)}$ be the solution to the pluriclosed flow starting from $\omega_0$. Then, $\omega(t)$ is Vaisman for all $t \in [0,T)$ if and only if $\omega_0$ has constant scalar curvature.
\end{thmint}

To prove this theorem, we use the fact that the {\it Lee vector field} $\theta_0^{\sharp_0}$ and the {\it anti-Lee vector field} $J\theta_0^{\sharp_0}$ of the initial metric $\omega_0$ are both holomorphic and Killing. Therefore, they give rise to a free action of a complex abelian Lie algebra $\mathfrak{t}$, which preserves both the complex structure $J$ and the initial Vaisman metric $\omega_0$ (see Section \ref{sect:prel} for the relevant definitions). Consequently, $\omega(t)$ is $\mathfrak{t}$-invariant for any $t \in [0,T)$. We remark that, in general, this Lie algebra action does not integrate to a compact Lie group action, so one cannot use the theory of principal bundles directly. However, the action of $\mathfrak{t}$ is enough to perform a {\it dimensional reduction} argument. Therefore, we obtain the result using characteristic classes, defined as in \cite[Proposition 7.5]{MR1362865}, and the fact that the transverse geometry has real dimension $2$. \smallskip

We remark that by \cite[Proposition 6.2]{NiZh24}, a Vaisman metric on a compact surface has constant scalar curvature if and only if its Bismut connection is {\it Ambrose-Singer}, namely, it has parallel curvature and parallel torsion. Consequently, such metrics are locally homogeneous, and so Theorem \ref{thm:MAIN} immediately implies the following.

\begin{corint} \label{cor:MAIN}
Given a compact Vaisman surface $(M^4, J, \omega_0)$, the solution $\{\omega(t)\}_{t \in [0,T)}$ to the pluriclosed flow starting from $\omega_0$ is Vaisman for all $t \in [0,T)$ if and only if it is locally homogeneous.
\end{corint}

This result implies that forcing the pluriclosed flow to evolve within the class of Vaisman metrics does not provide new insights into the behavior of its solutions. Indeed, the evolution of locally homogeneous metrics on surfaces is described in \cite{MR3511471}. However, even in the non-locally homogeneous setting, the abelian symmetry given by a Vaisman structure induces a dimensional reduction that may lead to a geometric characterization. This was accomplished in the case of $\mathsf{T}^2$-principal bundles over compact Riemann surfaces in \cite{MR4348696}.

Theorem \ref{thm:MAIN} also contributes to understanding whether the pluriclosed flow preserves additional geometric conditions, following the question raised in \cite{MR4257077}. Similar problems for Hermitian curvature flows have been addressed in the literature, {\it e.g.}, in \cite{MR4030526, MR4564028, MR4288257}. We recall that, in complex dimension $2$, Vaisman metrics are the pluriclosed metrics with {\it parallel Bismut torsion} \cite[Theorem 2]{MR4554474}, while in dimension $n >2$ compact Vaisman manifolds never admit pluriclosed metrics \cite[ Theorem 3.1]{angella-otiman}. Therefore, in this setting, the next natural step is understanding if Theorem \ref{thm:MAIN} generalizes to higher dimensional pluriclosed manifolds with parallel Bismut torsion, that have been recently classified in \cite{BPT24}. \medskip

The paper is organized as follows: in Section \ref{sect:prel} we collect some preliminary facts about abelian Lie algebra actions; Section \ref{sect:mainsect} contains the proof of Theorem \ref{thm:MAIN}.

\medskip
\section{Preliminaries on abelian Lie algebra actions}
\label{sect:prel} \setcounter{equation} 0

Let $(M^{2n}, J)$ be a compact complex manifold of complex dimension $n$ and denote by $\mathfrak{X}(M)$ the infinite-dimensional Lie algebra of vector fields on $M$. Assume that $\mathfrak{t} \subset \mathfrak{X}(M)$ is an abelian $2k$-dimensional subalgebra, with $k \leq n$, satisfying the following hypothesis:
\begin{itemize}
\item[$\bcdot$] $\mathfrak{t}$ {\it preserves $J$}, namely $\mathcal{L}_VJ=0$ for any $V \in \mathfrak{t}$;
\item[$\bcdot$] $\mathfrak{t}$ is {\it free}, namely the rank of the evaluation map
$$
{\rm ev}_x : \mathfrak{t} \to T_xM \,\, , \quad {\rm ev}_x(V) \coloneqq V_x
$$
is maximal for any $x \in M$;
\item[$\bcdot$] $\mathfrak{t}$ is {\it $J$-invariant}, namely $J\mathfrak{t} = \mathfrak{t}$.
\end{itemize}
Here, $\mathcal{L}$ denotes the Lie derivative. Notice that, by hypotheses, the complex structure $J$ on $M$ induces a linear complex structure on the Lie algebra $\mathfrak{t}$, that we still denote by $J$. Moreover, $\mathfrak{t}$ gives rise to a regular $2k$-dimensional distribution $\mathcal{V} \subset TM$ by setting $\mathcal{V}_x \coloneqq {\rm ev}_x(\mathfrak{t})$, which is called {\it vertical distribution}. By the Frobenius Theorem, $\mathcal{V}$ is integrable, and each maximal leaf of $\mathcal{V}$ is called {\it $\mathfrak{t}$-orbit}. We remark that, although $M$ is locally equivalent to a principal bundle, the orbit space $M/\mathfrak{t}$ can be highly non-smooth. \smallskip

We recall that a (possibly vector-valued) tensor field $\Phi$ of type $(r,s)$ on $M$ is said to be: {\it $\mathfrak{t}$-invariant} if $\mathcal{L}_V\Phi =0$ for any $V \in \mathfrak{t}$; {\it horizontal} if $s \geq 1$ and $V \lrcorner \Phi = 0$ for any $V \in \mathfrak{t}$; {\it basic} if it is both $\mathfrak{t}$-invariant and horizontal. Notice that the exterior differential ${\rm d}f$ of a $\mathfrak{t}$-invariant function $f$ is basic. Moreover, the exterior differential ${\rm d}\alpha$ of a basic $k$-form $\alpha$ is basic, and so this allows to define the {\it basic de Rham cohomology spaces} $H^k_{\rm d}(M)_{\mathfrak{t}}$ (see, {\it e.g.}, \cite[Section 7.1]{MR1362865} and references therein). \smallskip

A smooth $\mathfrak{t}$-valued $1$-form $\mu : TM \to \mathfrak{t}$ is called {\it principal connection} if it is $\mathfrak{t}$-invariant and verifies ${\rm ev}_x(\mu_x(V_x)) = V_x$ for any $V \in \mathfrak{t}$, $x \in M$ \cite[Section 3.1]{MR1362865}. Consequently, the kernel $\mathcal{H} \coloneqq {\rm ker}(\mu) \subset TM$ is a regular $2(n{-}k)$-dimensional distribution that is transverse to $\mathcal{V}$, which is called {\it horizontal distribution}. Fix now a basis $(V_1,{\dots},V_{2k})$ for $\mathfrak{t}$ satisfying $JV_{2i-1}=V_{2i}$ for any $1 \leq i \leq k$ and consider the corresponding splitting of $\mu$:
$$
\mu = \mu_1 \otimes V_1 + {\dots} + \mu_{2k} \otimes V_{2k} \,\, .
$$
Then, it is straightforward to check that the curvature $2$-forms
\begin{equation} \label{eq:defOmega}
\Omega_i \coloneqq {\rm d}\mu_i
\end{equation}
are basic. Moreover, by \cite[Proposition 7.5]{MR1362865}, the basic cohomology classes $[\Omega_i] \in H^2_{\rm d}(M)_{\mathfrak{t}}$ do not depend on the principal connection $\mu$, and are called {\it characteristic classes of $(M,J,\mathfrak{t})$.} \smallskip

Let now $g$ be a $\mathfrak{t}$-invariant Hermitian metric on $(M,J)$. Then, it induces a principal connection $\mu$ and an orthogonal splitting $TM = \mathcal{V} + \mathcal{H}$. We denote by $\hat{g}$ and $\check{g}$ the restriction of $g$ to $\mathcal{V}$ and $\mathcal{H}$, respectively. As in the case of principal bundles, the triple $(\hat{g},\mu,\check{g})$ determines uniquely the metric $g$. According to this splitting, the fundamental $2$-form $\omega \coloneqq g(J\cdot,\cdot)$ takes the form
$$
\omega = \sum_{1 \leq i < j \leq 2k} \lambda_{ij}\, \mu_i \wedge \mu_j + \check{\omega} \,\, ,
$$
where $\check{\omega} \coloneqq \check{g}(J\cdot,\cdot)$ is a basic $2$-form and $\lambda_{ij} \coloneqq \hat{g}(JV_i,V_j)$ are $\mathfrak{t}$-invariant functions. In the following, we will call {\it metric} both the tensors $g$ and $\omega$. When a $\mathfrak{t}$-invariant Hermitian metric is fixed, a vector field $X \in \mathfrak{X}(M)$ is called: {\it horizontal} if $X_x \in \mathcal{H}_x$ for any $x \in M$; {\it basic} if it is both $\mathfrak{t}$-invariant and horizontal.

\medskip
\section{The Vaisman condition along the pluriclosed flow}
\label{sect:mainsect} \setcounter{equation} 0

\subsection{Invariant geometry of compact Vaisman surfaces} \hfill \par

Let $(M^4, J)$ be a compact complex surface and $\mathfrak{t} \subset \mathfrak{X}(M)$ be an abelian $2$-dimensional subalgebra of holomorphic vector fields on $M$. We assume that $\mathfrak{t}$ is free and $J$-invariant. Fix also a basis $(V_1,V_2)$ for $\mathfrak{t}$ satisfying $JV_1 = V_2$.

Let now $g$ be a $\mathfrak{t}$-invariant Hermitian metric on $(M,J)$. Consider the splitting
\begin{equation} \label{eq:generalomega}
\omega = \lambda\,\mu_1 \wedge \mu_2 +\check{\omega}
\end{equation}
and the curvature $2$-forms $\Omega_1 = {\rm d}\mu_1$, $\Omega_2 = {\rm d}\mu_2$ as in Section \ref{sect:prel}.

\begin{lemma}
The $2$-form $\check{\omega}$ is closed, namely
\begin{equation} \label{eq:transK}
{\rm d}\check{\omega} = 0 \,\, .
\end{equation}
Moreover, there exists a unique choice of $\mathfrak{t}$-invariant, $\mathbb{R}$-valued, smooth functions $\sigma_1$, $\sigma_2$ on $M$ such that
\begin{equation} \label{eq:Omegasigma}
\Omega_1 = \sigma_1 \,\check{\omega} \,\, , \quad \Omega_2 = \sigma_2 \,\check{\omega} \,\, .
\end{equation}
In particular, the $2$-forms $\Omega_i$ are of type $(1,1)$, namely
\begin{equation} \label{eq:type(1,1)}
J\Omega_1 = \Omega_1 \,\, , \quad J\Omega_2 = \Omega_2 \,\, .
\end{equation}
\end{lemma}

\begin{proof}
The $3$-form ${\rm d}\check{\omega}$ is basic, and so $V \lrcorner {\rm d}\check{\omega} = 0$ for any $V \in \mathfrak{t}$. Moreover, if $X,Y,Z$ are horizontal vector fields, then ${\rm d}\check{\omega}(X,Y,Z) = 0$ for dimensional reasons. This proves \eqref{eq:transK}. Moreover, since $\dim(\mathcal{H}_x) = 2$ for any $x \in M$, \eqref{eq:Omegasigma} and \eqref{eq:type(1,1)} follow for dimensional reasons.
\end{proof}

We now list some properties of $\mathfrak{t}$-invariant pluriclosed metrics.

\begin{proposition} \label{prop:gpluriclosed}
Let $g$ be a $\mathfrak{t}$-invariant Hermitian metric on $(M^4,J)$. Then, the following claims hold true.
\begin{itemize}
\item[i)] The metric $g$ is pluriclosed if and only if $\lambda$ is constant.
\item[ii)] If $g$ is pluriclosed, then its Lee form $\theta$ is given by
\begin{equation} \label{eq:theta}
\theta = \lambda (\sigma_1\,\mu_2 -\sigma_2\,\mu_1) \,\, .
\end{equation}
\item[iii)] If $g$ is pluriclosed, then ${\rm d}\theta = 0$ if and only if both $\sigma_1$, $\sigma_2$ are constant.
\item[iv)] If $g$ is pluriclosed, then the class of $\check{\omega}$ in the basic cohomology space $H^2_{\rm d}(M)_{\mathfrak{t}}$ is non-trivial.
\end{itemize}
\end{proposition}

\begin{proof}
By \eqref{eq:defOmega}, \eqref{eq:generalomega} and \eqref{eq:transK}, the exterior differential ${\rm d}\omega$ is given by
\begin{equation} \label{eq:dw}
{\rm d}\omega = {\rm d}\lambda \wedge \mu_1 \wedge \mu_2 +\lambda\,\Omega_1 \wedge \mu_2 -\lambda\,\mu_1 \wedge \Omega_2
\end{equation}
and then, by \eqref{eq:type(1,1)}, we obtain
\begin{equation} \label{eq:Jdw}
J{\rm d}\omega = J{\rm d}\lambda \wedge \mu_1 \wedge \mu_2 -\lambda\, \mu_1 \wedge \Omega_1 -\lambda\, \mu_2 \wedge \Omega_2 \,\, .
\end{equation}
We observe that 
\begin{equation} \label{eq:OmegaOmega=0}
\Omega_1 \wedge \Omega_1 = \Omega_2 \wedge \Omega_2 = 0
\end{equation}
for dimensional reasons. Moreover, since ${\rm d}\lambda$ is basic, we obtain
\begin{equation} \label{eq:dlOmega=0}
{\rm d}\lambda \wedge \Omega_1 = J{\rm d}\lambda \wedge \Omega_1 = {\rm d}\lambda \wedge \Omega_2 = J{\rm d}\lambda \wedge \Omega_2 = 0
\end{equation}
for dimensional reasons. Therefore, since $\Omega_i$ are closed, by \eqref{eq:Jdw}, \eqref{eq:OmegaOmega=0} and \eqref{eq:dlOmega=0}, it follows that 
$$
{\rm d}J{\rm d}\omega = {\rm d}J{\rm d}\lambda \wedge \mu_1 \wedge \mu_2 \,\, .
$$
Then, $g$ is pluriclosed if and only if ${\rm d}J{\rm d}\lambda = 0$. Being $M$ compact, ${\rm d}J{\rm d}\lambda = 0$ if and only if $\lambda$ is constant, and this concludes the proof of claim $i)$.

Assume now that $\lambda$ is constant. Then, by \eqref{eq:Omegasigma} and \eqref{eq:dw}, we get
$$
{\rm d}\omega = \lambda\,\sigma_1\,\mu_2 \wedge \check{\omega} -\lambda\,\sigma_2\,\mu_1 \wedge \check{\omega}
$$
and so
\begin{equation} \label{eq:dwplur}
{\rm d}\omega = \lambda (\sigma_1\,\mu_2 -\sigma_2\,\mu_1) \wedge \omega \,\, .
\end{equation}
By uniqueness of the Lee form, \eqref{eq:theta} follows from \eqref{eq:dwplur}, and this concludes the proof of claim $ii)$. Furthermore, by \eqref{eq:Omegasigma} and \eqref{eq:theta}, we get
$$\begin{aligned}
{\rm d}\theta &= \lambda ({\rm d}\sigma_1 \wedge \mu_2 +\sigma_1 \Omega_2 -{\rm d}\sigma_2 \wedge \mu_1 -\sigma_2 \Omega_1) \\
&= \lambda ({\rm d}\sigma_1 \wedge \mu_2 -{\rm d}\sigma_2 \wedge \mu_1)
\end{aligned}$$
and so this proves claim $iii)$.

Finally, assume by contradiction that there exists a basic $1$-form $\eta$ on $M$ such that ${\rm d}\eta = \check{\omega}$. Being $\Omega_i$ basic forms, we have
$$
\Omega_1 \wedge \eta = \Omega_2 \wedge \eta = 0
$$
for dimensional reasons. Therefore, by \eqref{eq:generalomega}, we obtain
$$
\omega \wedge \omega = 2\lambda\, \mu_1 \wedge \mu_2 \wedge {\rm d}\eta +{\rm d}\eta \wedge {\rm d}\eta = {\rm d}\Big( 2\lambda\, \mu_1 \wedge \mu_2 \wedge \eta + \eta \wedge {\rm d}\eta\Big) \,\, ,
$$
which is impossible, being $M$ $4$-dimensional and compact. This completes the proof of claim $iv)$.
\end{proof}

We are now able to characterize when $\mathfrak{t}$-invariant Hermitian metrics are Vaisman.

\begin{proposition} \label{prop:Vaisman}
Let $g$ be a $\mathfrak{t}$-invariant Hermitian metric on $(M^4,J)$. Then, $g$ is Vaisman if and only if $\lambda$, $\sigma_1$ and $\sigma_2$ are constant.
\end{proposition}

\begin{proof}
Assume that $g$ is Vaisman. Then, since $\theta$ is Levi-Civita parallel, it is both closed and co-closed. Being $\theta$ co-closed, $g$ is pluriclosed for dimensional reasons. Then, by Proposition \ref{prop:gpluriclosed}, it follows that $\lambda$, $\sigma_1$ and $\sigma_2$ are constant.

Assume now that $\lambda$, $\sigma_1$ and $\sigma_2$ are constant. Then, a direct computation based on \eqref{eq:defOmega}, \eqref{eq:Omegasigma} and \eqref{eq:theta} shows that
\begin{gather}
\theta \wedge J\theta = \lambda^2(\sigma_1^2 +\sigma_2^2)\,\mu_1 \wedge \mu_2 \,\, , \label{eq:wVais1} \\
{\rm d}J\theta = -\lambda\,(\sigma_1^2 +\sigma_2^2)\,\check{\omega} \,\, . \label{eq:wVais2}
\end{gather}
Since $g(\mu_1,\mu_1) = g(\mu_2,\mu_2) = \lambda^{-1}$ and $g(\mu_1,\mu_2) = 0$, from \eqref{eq:theta} we have
\begin{equation} \label{eq:wVais3}
|\theta|^2 = \lambda\,(\sigma_1^2 +\sigma_2^2) \,\, .
\end{equation}
Therefore, by \eqref{eq:generalomega}, \eqref{eq:wVais1}, \eqref{eq:wVais2} and \eqref{eq:wVais3} we get
\begin{equation} \label{eq:potential}
|\theta|^2\omega = \theta \wedge J\theta -{\rm d}J\theta \,\, .
\end{equation}
Since $\lambda$, $\sigma_1$ and $\sigma_2$ are constant, \eqref{eq:theta} implies that the Lee vector field of $g$ is holomorphic. Then, by \cite[Proposition 1]{MR3927157} and \eqref{eq:potential}, $g$ is Vaisman. This concludes the proof.
\end{proof}

\begin{remark}
Notice that characterizing Vaisman metrics among $\mathfrak{t}$-invariant Hermitian metrics is not restrictive. Indeed, by \cite[Corollary 2.7]{MR1485384}, all the Vaisman metrics on a compact complex surface are invariant under the same $\mathfrak{t}$-action.
\end{remark}

\begin{remark}
We recall that a Hermitian metric on a compact complex surface is called {\it locally conformally K{\"a}hler} ({\it LCK} for short) if its Lee form $\theta$ is closed. Therefore, by Proposition \ref{prop:gpluriclosed} and Proposition \ref{prop:Vaisman}, it follows that a $\mathfrak{t}$-invariant Hermitian metric $g$ on a compact complex surface is Vaisman if and only if it is pluriclosed and LCK. While on compact surfaces the Vaisman condition implies $g$ being both pluriclosed and LCK, the converse generally is not true. Indeed, up to conformal change, a LCK metric on a compact surface can be chosen to be pluriclosed, but not all LCK surfaces admit a Vaisman metric.
\end{remark}

Finally, we recall that the Bismut Ricci form of Vaisman surfaces is horizontal. Namely, the following result holds true.

\begin{lemma}[c.f.\ Lemma 3.13 of \cite{MR4480223} and (4.4) of \cite{BPT24}]
Let $g$ be a $\mathfrak{t}$-invariant Hermitian metric on $(M^4,J)$. If $g$ is Vaisman, then its Bismut Ricci form $\rho$ and its Bismut scalar curvature $s$ verify
\begin{equation} \label{eq:rhohor}
\rho = s\, \check{\omega} \,\, .
\end{equation}
\end{lemma}

\subsection{Pluriclosed flow solutions that are Vaisman for all times} \hfill \par

Let $(M^4, J, \omega_0)$ be a compact Vaisman surface, $\{\omega(t)\}_{t \in [0,T)}$ the solution to the pluriclosed flow starting from $\omega_0$ and denote by $g(t)$ the associated symmetric tensors. We denote by $\theta(t)$ the Lee form of $\omega(t)$. Since $\omega_0$ is Vaisman, the vector fields 
\begin{equation} \label{eq:choiceVi}
V_1 \coloneqq J\theta(0)^{\sharp_{g(0)}} \,\, , \quad V_2 \coloneqq -\theta(0)^{\sharp_{g(0)}}
\end{equation}
generate an abelian, $2$-dimensional Lie algebra $\mathfrak{t}$ of holomorphic vector fields on $(M,J)$ that is free and $J$-invariant. Moreover, being $\omega_0$ $\mathfrak{t}$-invariant, it follows that $\omega(t)$ is $\mathfrak{t}$-invariant for all $t \in [0,T)$. Accordingly, we write
\begin{equation} \label{eq:generalomega(t)}
\omega(t) = \lambda(t)\,\mu_1(t) \wedge \mu_2(t) +\check{\omega}(t) \,\, ,
\end{equation}
where $\lambda(t) = g(t)(V_1,V_1)$, $\mu_i(t)$ are the components of the principal connection $\mu(t)$ induced by $g(t)$ with respect to the basis $(V_1,V_2)$ and $\check{\omega}(t)$ is the restriction of $\omega(t)$ to $\mathcal{H}(t) = \mathcal{V}^{\perp_{g(t)}}$. We also denote by $s(t)$ the Bismut scalar curvature of $\omega(t)$, by $\Omega_i(t) = {\rm d}\mu_i(t)$ the curvature $2$-forms of $\mu(t)$ and by $\sigma_i(t)$ the functions defined in \eqref{eq:Omegasigma}.

\begin{proof}[Proof of Theorem \ref{thm:MAIN}]
If $\omega_0$ has constant scalar curvature, then $\omega(t)$ is Vaisman for all $t \in [0,T)$ by \cite[Theorem B]{MR4257077}.
Assume now that $\omega(t)$ is Vaisman for all $t \in [0,T)$. According to Proposition \ref{prop:Vaisman}, $\lambda(t)$, $\sigma_1(t)$ and $\sigma_2(t)$ are constant on $M$ for all $t \in [0,T)$. Moreover, by \cite[Proposition 7.5]{MR1362865}, the basic cohomology classes $[\Omega_1(t)]$, $[\Omega_2(t)]$ do not depend on $t$, and so, by \eqref{eq:choiceVi}, we obtain:
\begin{align}
\sigma_1(t) \neq 0 \quad \text{for all $t \in [0,T)$} \,\, , \label{eq:sigma1t} \\
\sigma_2(t) = 0 \quad \text{for all $t \in [0,T)$} \,\, . \label{eq:sigma2t}
\end{align}
By the evolution equation of the pluriclosed flow \eqref{eq:PCflow}, \eqref{eq:generalomega(t)} and \eqref{eq:rhohor}, we get
\begin{gather}
\partial_t\, \big(\lambda(t)\,\mu_1(t) \wedge \mu_2(t)\big) = 0 \,\, , \label{eq:splitflow1} \\
\partial_t\, \check{\omega}(t) = -s(t)\, \check{\omega}(t) \,\, . \label{eq:splitflow2}
\end{gather}
By \eqref{eq:splitflow1}, we get
$$
\lambda'(t)\,\mu_1(t) \wedge \mu_2(t) +\lambda(t)\big(\partial_t\mu_1(t) \wedge \mu_2(t) +\mu_1(t) \wedge \partial_t\mu_2(t)\big) = 0
$$
and so, since
$$
g(t)\big(\mu_1(t) \wedge \mu_2(t), \mu_1(t) \wedge \mu_2(t)\big) = \lambda(t)^{-2} \,\, ,
$$
$$
g(t)\big(\partial_t\mu_1(t) \wedge \mu_2(t), \mu_1(t) \wedge \mu_2(t)\big) = \lambda(t)^{-1} g(t)\big(\partial_t\mu_1(t), \mu_1(t)\big) \,\, ,
$$
we obtain
\begin{equation} \label{eq:lambda'}
g(t)\big(\partial_t\mu_1(t),\mu_1(t)\big) = -\tfrac12\lambda'(t)\lambda(t)^{-2} \,\, .
\end{equation}
Moreover, since ${\rm d}$ and $\partial_t$ commute, by \eqref{eq:sigma2t}, \eqref{eq:splitflow1} and \eqref{eq:splitflow2}, we obtain
$$\begin{aligned}
0 &= \partial_t\, {\rm d} \big(\lambda(t)\,\mu_1(t) \wedge \mu_2(t)\big) \\
&= \partial_t\, \big(\lambda(t)\,\Omega_1(t) \wedge \mu_2(t) -\lambda(t)\,\mu_1(t) \wedge \Omega_2(t)\big) \\
&= \partial_t\, \big(\lambda(t)\,\sigma_1(t)\,\mu_2(t) \wedge \check{\omega}(t)\big) \\
&= \lambda'(t)\,\sigma_1(t)\,\mu_2(t) \wedge \check{\omega}(t)
+\lambda(t)\big(\sigma'_1(t)\,\mu_2(t) +\sigma_1(t)\,\partial_t\mu_2(t)\big) \wedge \check{\omega}(t) +\lambda(t)\,\sigma_1(t)\,\mu_2(t) \wedge \partial_t\check{\omega}(t) \\
&= \Big(\lambda'(t)\,\sigma_1(t)\,\mu_2(t) +\lambda(t)\big(\sigma'_1(t)\,\mu_2(t) +\sigma_1(t)\,\partial_t\mu_2(t) -\sigma_1(t)\,s(t)\,\mu_2(t) \big)\Big) \wedge \check{\omega}(t)
\end{aligned}$$
Therefore, since $\check{\omega}(t) \neq 0$, we get
\begin{equation} \label{eq:angle12}
\lambda'(t)\,\sigma_1(t)\,\mu_2(t) +\lambda(t)\big(\sigma'_1(t)\,\mu_2(t) +\sigma_1(t)\,\partial_t\mu_2(t) -\sigma_1(t)\,s(t)\,\mu_2(t) \big) = 0 \,\, .
\end{equation}
By contracting with $\mu_2(t)$ and recalling that $g(t)(\mu_i(t),\mu_i(t)) = \lambda(t)^{-1}$, we get
$$
\lambda'(t)\lambda(t)^{-1}\sigma_1(t) +\sigma'_1(t) +\lambda(t)\sigma_1(t)\,g(t)(\partial_t\mu_2(t), \mu_2(t)) -\sigma_1(t)s(t) = 0 \,\, .
$$
Finally, by \eqref{eq:lambda'}, we obtain
\begin{equation} \label{eq:lambda''}
\tfrac12\lambda'(t)\,\sigma_1(t) +\lambda(t)\,\sigma'_1(t) -\lambda(t)\sigma_1(t)s(t) = 0 \,\, .
\end{equation}
Being $\lambda(t)$ and $\sigma_1(t)$ constant and non-vanishing on $M$, it follows that $s(t)$ is constant on $M$ for all $t \in [0,T)$, and this concludes the proof.
\end{proof}

With these techniques based on dimensional reduction and the independence of characteristic classes, it is possible to further describe the evolution of a Vaisman metric with constant scalar curvature on a compact surface along the pluriclosed flow. More precisely, the following computations show that the flow only affects the transversal geometry. To see this, assume that $\omega(t)$ is a solution of the pluriclosed flow that is Vaisman for all $t \in [0,T)$. Then, since the basic cohomology class $[\Omega_1(t)]$ does not depend on $t$, we get
$$\begin{aligned}
0 &= \partial_t [\Omega_1(t)] \\
&= \big[\partial_t (\sigma_1(t)\,\check{\omega}(t))\big] \\
&= \sigma_1'(t)[\check{\omega}(t)] +\sigma_1(t)[\partial_t\,\check{\omega}(t)] \\
&= \sigma_1'(t)[\check{\omega}(t)] -\sigma_1(t)s(t)[\check{\omega}(t)] \,\, .
\end{aligned}$$
Since $[\check{\omega}(t)] \neq 0$ by Proposition \ref{prop:gpluriclosed}, we get
\begin{equation} \label{eq:sigma1'}
\sigma_1'(t) =\sigma_1(t)\,s(t) \,\, .
\end{equation}
Therefore, by \eqref{eq:sigma1t}, \eqref{eq:lambda''} and \eqref{eq:sigma1'}, we obtain
$$
\lambda'(t) = 0
$$
and so $\lambda(t)$ is also constant in time. Finally, by \eqref{eq:splitflow1}, we get
$$
\partial_t\mu_1(t) \wedge \mu_2(t) +\mu_1(t) \wedge \partial_t\mu_2(t) = 0
$$
and so, by \eqref{eq:lambda'} and \eqref{eq:angle12}, we conclude that 
$$
\partial_t\mu_1(t) = \partial_t\mu_2(t) = 0 \,\, .
$$
Therefore, the principal connection $\mu$ and the horizontal distribution $\mathcal{H}$ do not depend on $t$.

\begin{remark}
The fact that $\mu$ and $\mathcal{H}$ do not vary with time for a Vaisman solution to the pluriclosed flow can also be deduced by the construction of an explicit solution as in \cite[Proof of Theorem B]{MR4257077}.
\end{remark}

\end{document}